\documentclass[final, nomarks]{dmtcs-episciences}

\usepackage[utf8]{inputenc}
\usepackage{subfigure}

\usepackage{amsmath}
\usepackage{amsthm}
\usepackage{oubraces}

\def\A{\mathcal A}
\def\uu{\mathbf{u}}
\def\vv{\mathbf{v}}

\theoremstyle{definition}
\newtheorem{definition}{Definition}
\newtheorem{corollary}[definition]{Corollary}
\newtheorem{remark}[definition]{Remark}
\newtheorem{example}[definition]{Example}

\theoremstyle{plain}
\newtheorem{theorem}[definition]{Theorem}
\newtheorem{proposition}[definition]{Proposition}
\newtheorem{lemma}[definition]{Lemma}
\newtheorem{observation}[definition]{Observation}

\author[L. Dvo\v r\'akov\'a and V. Hendrychov\'a]{Lubom\'ira Dvo\v r\'akov\'a \thanks{Supported by the Ministry of Education, Youth and Sports of the Czech Republic through the project CZ.02.1.01/0.0/0.0/16\_019/0000778.} \and Veronika Hendrychov\'a\thanks{Supported by Czech Technical University in Prague, through the project SGS23/187/OHK4/3T/14.}}
  \title{String attractors of Rote sequences}
\affiliation{
  Czech Technical University in Prague, Czech Republic}

\begin{document}
\publicationdata{vol. 26:3}{2024}{2}{10.46298/dmtcs.12385}{2023-10-10; 2023-10-10; 2024-03-04; 2024-07-18}{2024-07-23}

\maketitle

\begin{abstract}
  In this paper, we describe minimal string attractors of pseudopalindromic prefixes of standard complementary-symmetric Rote sequences. Such a class of Rote sequences forms a subclass of binary generalized pseudostandard sequences, \textit{i.e.}, of sequences obtained when iterating palindromic and antipalindromic closures. When iterating only palindromic closure, palindromic prefixes of standard Sturmian sequences are obtained and their string attractors are of size two. However, already when iterating only antipalindromic closure, antipalindromic prefixes of binary pseudostandard sequences are obtained and we prove that the minimal string attractors are of size three in this case. We conjecture that the pseudopalindromic prefixes of any binary generalized pseudostandard  sequence have a minimal string attractor of size at most four.  

\end{abstract}

\section{Introduction}
In the last years, significant attention has been dedicated to the study of string attractors. Their definition and first results were provided by Kempa and Prezza~\cite{KempaPrezza2018}: a \emph{string attractor} of a finite word $w=w_0w_1\cdots w_{n-1}$, where $w_i$ are letters, is a subset $\Gamma$ of $\{0,1,\dots, n-1\}$ such that each non-empty factor of $w$ has an occurrence containing an element of $\Gamma$. They are closely related to methods of compression of highly repetitive data, so-called dictionary compressors. On one hand, it was shown by Kempa and Prezza~\cite{KempaPrezza2018} that dictionary compressors can be interpreted as approximation algorithms for the smallest string attractors, due to the measures induced by the compressors being lower bounded by the smallest string attractor size. On the other hand, attractors are able to express bounds for and potentially unify the dictionary compression measures. 

However, the general problem of finding the minimum size of a string attractor is NP-complete. It is therefore natural to study the problem in the context of combinatorics on words and to put certain restrictions on the input, which makes the computation tractable.

In this paper we carry on the study of string attractors of important classes of sequences. 
Attractors of minimum size of factors/prefixes/particular prefixes of the following sequences have been determined so far: standard Sturmian sequences~\cite{Mantaci2021, Dv2023}, the Tribonacci sequence~\cite{Shallit2021}, episturmian sequences~\cite{Dv2023}, the Thue-Morse sequence~\cite{Kutsukake2020, Shallit2021, Dolce2023}, the period-doubling sequence~\cite{Shallit2021}, the powers of two sequence~\cite{Shallit2021, Kociumaka2021}.
Recently, string attractors of fixed points of $k$-bonacci-like morphisms have been described~\cite{GhRoSt2023}.

When studying string attractors of episturmian sequences~\cite{Dv2023}, we could see usefulness of palindromic closures. It is thus natural to ask as the next question what can be said about string attractors of binary generalized pseudostandard sequences (defined in~\cite{LuDeLu}), \textit{i.e.}, when iterating palindromic and antipalindromic closures.
Two subclasses of generalized pseudostandard sequences have been already studied: standard Sturmian sequences and the Thue-Morse sequence. 
Here, we consider pseudostandard sequences, \textit{i.e.}, binary sequences obtained when iterating solely the antipalindromic closure, and complementary-symmetric Rote sequences. We show that the minimal string attractor of antipalindromic prefixes of pseudostandard sequences is of size three, while the minimal string attractor of pseudopalindromic prefixes of complementary-symmetric Rote sequences is of size two. In the latter case, the description of Rote sequences from~\cite{MaPa} plays an important role.  Based on computer experiments, we conjecture that the minimum size of string attractors of pseudopalindromic prefixes of binary generalized peudostandard sequences is at most four.

\section{Preliminaries}
\label{Section_Preliminaries}
An \emph{alphabet} $\A$ is a finite set of symbols called \emph{letters}.
A \emph{word} over $\A$ of \emph{length} $n$ is a string $u = u_0 u_1 \cdots u_{n-1}$, where $u_i \in \A$ for all $i \in \{0,1, \ldots, n-1\}$. We let $|u|$ denote the length of $u$.
The set of all finite words over $\A$ together with the operation of concatenation forms a monoid, denoted $\A^*$.
Its neutral element is the \emph{empty word} $\varepsilon$ and we write $\A^+ = \A^* \setminus \{\varepsilon\}$.
If $u = xyz$ for some $x,y,z \in \A^*$, then $x$ is a \emph{prefix} of $u$, $z$ is a \emph{suffix} of $u$ and $y$ is a \emph{factor} of $u$.

A \emph{sequence} over $\A$ is an infinite string $\uu = u_0 u_1 u_2 \cdots$, where $u_i \in \A$ for all $i \in \naturals$. 
We always denote sequences by bold letters. 

A sequence $\uu$ is \emph{eventually periodic} if $\uu = vwww \cdots = v(w)^\omega$ for some $v \in \A^*$ and $w \in \A^+$. If $\uu$ is not eventually periodic, then it is \emph{aperiodic}.
A~\emph{factor} of $\uu = u_0 u_1 u_2 \cdots$ is a word $y$ such that $y = u_i u_{i+1} u_{i+2} \cdots u_{j-1}$ for some $i, j \in \naturals$, $i \leq j$. If $i=j$, then $y=\varepsilon$.
In the context of string attractors, the set $\{i, i+1,\dots, j-1\}$ is called an \emph{occurrence} of the factor $y$ in $\uu$. (Usually, only the number $i$ is called an occurrence of $y$ in $\uu$.)
If $i=0$, the factor $y$ is a~\emph{prefix} of $\uu$.
A factor $w$ of $\uu$ is \emph{left special} if $aw, bw$ are factors of $\uu$ for at least two distinct letters $a,b \in \A$. A sequence $\uu$ is said to be \emph{closed under reversal} if for each factor $w=w_0 w_1\cdots w_{n-1}$, where $w_i \in \mathcal A$, $\uu$ contains also its mirror image $w_{n-1}\cdots w_1w_0$.
A binary sequence $\uu$ is called \emph{Sturmian} if $\uu$ is closed under reversal and $\uu$ contains exactly one left special factor of each length. If moreover each left special factor is a prefix of $\uu$, then $\uu$ is \emph{standard Sturmian}.

A \emph{string attractor} (or \emph{attractor} for short) of a word $w=w_0w_1\cdots w_{n-1}$, where $w_i\in \mathcal A$, is a set $\Gamma \subset \{0,1,\dots, n-1\}$ such that every non-empty factor of $w$ has an occurrence in $w$ containing at least one element of $\Gamma$. If $i \in \Gamma$ and a word $f$ has an occurrence in $w$ containing $i$, we say that $f$ \emph{crosses} $i$ and we also say that $f$ \emph{ crosses the attractor} $\Gamma$. For instance, $\Gamma=\{1,3\}$ is an attractor of $w=\tt 0\underline{1}0\underline{0}10$ (it corresponds to the underlined positions). The factor $\tt 00$ crosses the position 3 and thus it crosses the attractor. $\Gamma$ is an attractor of minimum size -- minimal attractor for short -- since each attractor necessarily contains occurrences of all distinct letters of the word.

\section{Palindromic and antipalindromic closures}
Throughout the paper, we deal only with binary sequences. Therefore we define the notions of reversal and palindrome over a binary alphabet, too. 
\label{sec:closures}
\begin{definition}
The map $R: \{\tt 0, \tt 1\}^* \to \{\tt 0, \tt 1\}^*$, called \emph{reversal}, associates with each word its mirror image, \textit{i.e.}, $R(w_0w_1\cdots w_{n-1})=w_{n-1}\cdots w_1 w_0$, where $w_i \in \{\tt 0, \tt 1\}$ for each $i \in \{0,1,\dots, n-1\}$.
The map $E: \{\tt 0, \tt 1\}^* \to \{\tt 0, \tt 1\}^*$, called \emph{exchange antimorphism}, is a composition of reversal and letter exchange, \textit{i.e.}, $E(w_0w_1\cdots w_{n-1})=\overline{w_{n-1}}\cdots \overline{w_1} \ \overline{w_0}$, where $w_i \in \{\tt 0, \tt 1\}$  and $\overline{\tt 0}=\tt 1$ and $\overline{\tt 1}=\tt 0$. 
A~word $w \in \{\tt 0, \tt 1\}^*$ is a \emph{palindrome} if $w=R(w)$ and $w$ is an \emph{antipalindrome} if $w=E(w)$. A word is a \emph{pseudopalindrome} if it is a palindrome or an antipalindrome. 

Consider $w \in \{\tt 0, \tt 1\}^*$. Then $w^{R}$ is the shortest palindrome having $w$ as prefix and it is called the \emph{palindromic closure} of $w$. Similarly, $w^E$ is the shortest antipalindrome having $w$ as prefix and it is called the \emph{antipalindromic closure} of $w$. The \emph{pseudopalindromic closure} is a term covering both palindromic and antipalindromic closure. 
\end{definition}
Let $w \in \{\tt 0, \tt 1\}^{+}$. It follows immediately from the definition that $w^{R}=vxR(v)$, where $w=vx$ and $x$ is the longest palindromic suffix of $w$.
Similarly, $w^{E}=vyE(v)$, where $y$ is the longest antipalindromic suffix of $w$ (possibly empty).
\begin{example}
We have $({\tt 000})^R=\tt 000$, $({\tt 000})^E=\tt 000111$, $({\tt 0001})^{R}=\tt 0001000$, $({\tt 0001})^E=\tt 000111$, $({\tt 01101})^R=\tt 0110110$, $({\tt 01101})^{E}=\tt 01101001$.
\end{example}
\begin{definition}\label{def:generalized_pseudostandard}
Let $\Delta = \delta_1 \delta_2 \cdots$ and $\Theta = \vartheta_1 \vartheta_2 \cdots$, where $\delta_i \in \{\tt 0,1\}$ and $\vartheta_i \in \{E, R\}$ for all $i \in \naturals, i\geq 1$. The sequence $\mathbf{u}(\Delta, \Theta)$, called \emph{generalized pseudostandard sequence}, is the sequence having prefixes $w_n$ obtained from the recurrence relation
\begin{displaymath} w_{n+1} = (w_n \delta_{n+1})^{\vartheta_{n+1}},\end{displaymath}
\begin{displaymath} w_0 = \varepsilon.\end{displaymath} The bi-sequence $(\Delta, \Theta)$ is called the \emph{directive bi-sequence} of the word $\mathbf{u}(\Delta, \Theta)$.
\end{definition}
\begin{example}
Consider $\uu=\uu(\Delta, \Theta)$ with $\Delta={\tt 0}{\tt 1}^{\omega}$ and $\Theta=(RE)^{\omega}$, then $\uu$ is the Thue-Morse sequence~\cite{LuDeLu}. Here are the first six pseudopalindromic prefixes:
\begin{displaymath}
\begin{array}{rcll}
  w_0&=&\varepsilon \\
  w_1&=&{\tt 0}\\
  w_2&=&{\tt 01} \\
  w_3&=&{\tt 0110}\\
  w_4&=&{\tt 01101001}\\
  w_5&=&{\tt 0110100110010110}\,.
\end{array}
\end{displaymath}

\end{example}

For some special forms of $(\Delta, \Theta)$, well-known classes of sequences are obtained:
\begin{enumerate}
    \item $\uu$ is a \emph{standard Sturmian sequence} if $\uu=\mathbf{u}(\Delta, \Theta)$ for some $\Delta$ containing both letters infinitely many times and $\Theta=R^{\omega}$.
    \item $\uu$ is a \emph{pseudostandard sequence} if $\uu=\mathbf{u}(\Delta, \Theta)$ for some $\Delta$ and $\Theta=E^{\omega}$.
\end{enumerate}

In the sequel, when examining Rote sequences, we will need the following statement about the form of palindromic prefixes of a standard Sturmian sequence.~\footnote{The palindromic prefixes of standard Sturmian sequences are also known as central words~\cite{Lothaire}.}
\begin{proposition}[Proposition 7~\cite{LuMi94}]\label{prop:Sturmian_palindromes}
Let $\uu$ be a standard Sturmian sequence and let $(u_n)_{n=0}^{\infty}$ be the sequence of palindromic prefixes of $\uu$ ordered by length. If $u_n$ contains both letters, then for some $a\in \{\tt 0, \tt 1\}$
\begin{displaymath}
u_n=u_{n-1}a\overline{a}u\,,
\end{displaymath}
where $u$ is the longest palindromic prefix of $u_n$ followed by $\overline{a}$ (\textit{i.e.}, $u\overline a$ is a prefix of $u_n$).
\end{proposition}

\section{String attractors of Sturmian sequences}
String attractors of palindromic prefixes of episturmian sequences were described in~\cite{Dv2023} (Theorem 7). Let us recall here the statement restricted to the binary alphabet, together with its proof.  Similar ideas will be used for pseudostandard sequences.

\begin{theorem}\label{thm:sturmian_attractor}\cite{Dv2023}
Let $v$ be a non-empty palindromic prefix of a standard Sturmian sequence. 
For every letter $a$ occurring in $v$,  denote 
\begin{displaymath}
r_a=\max\{|p| \, : \, p \emph{\  is a palindrome  and }  pa \emph{ is a prefix of} \ v\}.
\end{displaymath}
Then $\Gamma = \{r_a: a \emph{ occurs in  } v\} $ is an attractor of $v$ and it is of minimum size.
\end{theorem}

\begin{proof}
To construct a standard Sturmian sequence, we use only palindromic closure in Definition~\ref{def:generalized_pseudostandard} and each palindromic prefix $v$ is equal to $w_n$ for some $n \in \naturals$. Let us assume that the first letter of $\Delta$ is $\tt 0$. 
We will prove the statement by mathematical induction on $n$. Let us recall that we index positions from $0$, \textit{i.e.}, $v=v_0 v_1\cdots v_{|v|-1}$.
\begin{itemize}
    \item For $n=1$ we have $w_1=\tt 0$ and its attractor equals $\{0\}$. The longest palindromic prefix of $w_1$ followed by $\tt 0$ is equal to $w_0=\varepsilon$ and its length satisfies $|w_0|=0$.
    \item For $n\geq 2$ we assume that $w_{n-1}$ has an attractor of the form from the statement. We have $w_n=(w_{n-1}{a})^{R}$ for some $a \in \{\tt 0,1\}$. The following three situations may occur:
    \begin{enumerate}
        \item $w_n=w_{n-1}a$: According to the definition of palindromic closure, this happens only for $a=\tt 0$ and $w_{n-1}={\tt 0}^{n-1}$. The longest palindromic prefix of $w_n={\tt 0}^{n}$ followed by $\tt 0$ is $w_{n-1}={\tt 0}^{n-1}$. The length of $w_{n-1}$ is $n-1$ and, indeed, $\{n-1\}$ is an attractor of $w_n$.
        \item $w_n=w_{n-1}aw_{n-1}$: By the definition of palindromic closure, this happens only in case when $w_{n-1}={\tt 0}^{n-1}$ and $a=\tt 1$. Then $w_n={\tt 0}^{n-1}{\tt 1}{\tt 0}^{n-1}$ and $\Gamma=\{n-2, n-1\}$ is indeed an attractor of $w_n$.
        \item $w_n=w_{n-1}au$ for some $u \not =\varepsilon$ and $u \not =w_{n-1}$: Then $w_{n-1}$ contains both letters. We want to prove that $\Gamma=\{r_{\tt 0},  r_{\tt 1}\}$, as defined in the statement, is an attractor of $w_n$. Since the longest palindromic prefix of $w_n$ followed by $b$, where $b \in \{\tt 0,1\}$, $b\not =a$, is the same as in $w_{n-1}$, we know by induction assumption that $\{r_{b}, r'_{a}\}$ is an attractor of $w_{n-1}$, where $r'_{a}=|w_\ell|$ and $w_\ell$ is the longest palindromic prefix of $w_{n-1}$ followed by $a$.
        By the definition of palindromic closure we have 
        
        \begin{equation}\label{eq:w_n}
       w_n=\underbrace{R(u){a}w_\ell}_{w_{n-1}}{\underline{a}}u=R(u){a}\underbrace{w_\ell{\underline{a}}u}_{w_{n-1}}\,.
        \end{equation}
        
       Then each factor $f$ of $w_n$ either has an occurrence containing the position $|w_{n-1}|$, \textit{i.e.}, corresponding to the second (underlined) $a$ in~(\ref{eq:w_n}), or $f$ is entirely contained in $w_{n-1}$. In the latter case, $f$ has an occurrence crossing the attractor of $w_{n-1}$. Thus, $f$ either crosses the position $r_b$ or $f$ crosses the position $r'_{a}=|w_\ell|$. In the first case, we are done since $f$ crosses $\Gamma$. In the second case, according to~(\ref{eq:w_n}), the factor $f$ has also an occurrence in $w_n$ containing the position $r_{a}=|w_{n-1}|$ (corresponding to the underlined $a$). To sum up, we have proved that each factor of $w_n$ has an occurrence crossing $\{r_{a}, r_{b}\}=\{r_{\tt 0}, r_{\tt 1}\}=\Gamma$.
       
    \end{enumerate}
    \end{itemize}
\end{proof}

\begin{example}
\label{ex:FiboDef}
The most famous standard Sturmian sequence is the \emph{Fibonacci sequence}
\begin{displaymath}
\mathbf u = \mathbf u(\Delta, R^{\omega})\,,
\end{displaymath}
where $\Delta=(\tt 01)^{\omega}$.
The first six palindromic prefixes of $\mathbf u$ with the positions of the attractor from Theorem~\ref{thm:sturmian_attractor} underlined read:
\begin{displaymath}
\begin{array}{rcll}
  w_0&=&\varepsilon \\
  w_1&=&\underline{\tt 0}\\
  w_2&=&\underline{\tt 01}\tt 0 \\
  w_3&=&\tt 0\underline{\tt 1}\tt 0\underline{\tt 0}\tt 10\\
  w_4&=&\tt 010\underline{\tt 0}\tt 10\underline{\tt 1}\tt 0010\\
  w_5&=&\tt 010010\underline{\tt 1}\tt 0010\underline{\tt 0}\tt 1010010\,.
\end{array}
\end{displaymath}
\end{example}

\section{String attractors of pseudostandard sequences}
As a new result we will describe string attractors of antipalindromic prefixes of pseudostandard sequences. 
\begin{theorem} \label{thm:pseudostandard}
Let $v$ be a non-empty antipalindromic prefix of a pseudostandard sequence starting with the letter $\tt 0$. For every letter $a$ occurring in $v$,  denote 
\begin{displaymath}
e_a=\max\{|q| \, : \, q \emph{\  is an antipalindrome  and }  qa \emph{ is a prefix of} \ v\}.
\end{displaymath}
If such a prefix does not exist, then set $e_a=e_{\overline a}$.\\
Then $\Gamma = \{e_{\tt 0}, e_{\tt 1}, |v|-e_{\tt 1}-1\}$ is an attractor of $v$.

Moreover, 
if $v=w_n$, $n\geq 2$, from Definition~\ref{def:generalized_pseudostandard}, then 
\begin{itemize}
\item $\Gamma$ is of size two if and only if $\Delta$ starts with ${\tt 0}^{n}$;
\item $\Gamma$ is a minimum size attractor if and only if $\Delta$ does not start with ${\tt 0}{\tt 1}^{n-1}$. 
\end{itemize}
\end{theorem}

\begin{proof}
To construct a pseudostandard sequence $\uu$, we use only antipalindromic closure in Definition~\ref{def:generalized_pseudostandard} and each antipalindromic prefix $v$ is equal to $w_n$ for some $n \in \naturals$.  
We will prove the statement about the form of attractors of $w_n$ by mathematical induction on $n$. Let us recall that we index positions from $0$, \textit{i.e.}, $v=v_0 v_1\cdots v_{|v|-1}$.

\begin{itemize}
\item If ${\tt 0}^n$ is a prefix of $\Delta$, then we have $w_n=({\tt 01})^{n}$ and $e_{\tt 0}=2n-2$ and $e_{\tt 1}=e_{\tt 0}$. Then $\Gamma = \{e_{\tt 0},|w_n|-1-e_{\tt 0}\}=\{2n-2, 1\}$ is indeed an attractor of $w_n=\underline{\tt 0}\underline{\tt 1}$ for $n=1$, resp. $w_n={\tt 0}\underline{\tt 1}({\tt 01})^{n-2}\underline{\tt 0}{\tt 1}$ for $n \geq 2$ (we underlined the positions of $\Gamma$). In particular,
this proves the result for all $n\geq 1$ when $\Delta={\tt 0}^{\omega}$.
    \item Now, assume ${\tt 0}^{k}\tt 1$ is a prefix of $\Delta$ for some $k \geq 1$. Then $w_{k+1}= ({\tt 01})^{k}{\tt 10}({\tt 01})^{k}$, $e_{\tt 0}=|({\tt 01})^{k-1}|=2k-2$, and $e_{\tt 1}=|({\tt 01})^k|=2k$. It is readily seen that $\Gamma=\{2k-2, 2k, 2k+1\}$ is an attractor of $w_{k+1}=({\tt 01})^{k-1} \underline{\tt 0}{\tt 1}\underline{\tt 1}\underline{\tt 0}({\tt 01})^{k}$ (we underlined the positions of $\Gamma$). Using the first item, each $w_n$, for $1\leq n\leq k$, also has an attractor of the form from the statement. 
    \item For $n> k+1$ we assume that $w_{n-1}$ has an attractor of the form from the statement. The following two situations may occur:
    \begin{enumerate}
    \item $\delta_n=1$ and 
      $w_n=(w_{n-1}{\tt 1})^{E}=v {\tt 0} w_\ell {\tt 1} E(v)$,
            where $w_{\ell} \not= \varepsilon$, $w_{\ell}$ is the longest antipalindromic prefix followed by $\tt 1$ in $w_{n-1}$, and $v$ is a prefix of the word $w_n$. 
            The longest antipalindromic prefix $q$ followed by $\tt 0$ in $w_{n-1}$ and in $w_n$ is the same ($q$ may be empty if $k=1$). The longest antipalindromic prefix of $w_n$ followed by $\tt 1$ is equal to $w_{n-1}=v{\tt 0}w_{\ell}$. We assume that $\{|q|, |w_\ell|, |w_{n-1}|-|w_\ell|-1\}$ is an attractor of $w_{n-1}$ and we want to show that $\Gamma=\{|q|, |w_{n-1}|, |w_n|-|w_{n-1}|-1\}$ is then an attractor of $w_n$. Below, we underline the positions of the attractor of $w_{n-1}$ and the positions of $\Gamma$ in $w_n$:
            \begin{displaymath}
            \begin{array}{rcl}
            w_{n-1}&=&q\underline{\tt 0}\cdots=v \underline{\tt 0} w_\ell=w_\ell \underline{\tt 1}E(v)\,;\\
            w_n&=&q\underline{\tt 0}\cdots=v \underline{\tt 0} w_\ell \underline{\tt 1} E(v)\,.
        \end{array}
        \end{displaymath}
        Each factor $f$ of $w_n$ either crosses $|w_{n-1}|$, \textit{i.e.}, the underlined $\tt 1$, or $f$ is contained in $w_{n-1}$. By the form of the attractor of $w_{n-1}$, the factor $f$ crosses $|q|$ or $|v|=|w_{n-1}|-|w_\ell|-1=|w_n|-|w_{n-1}|-1$ or $|w_\ell|$. In the last case, since $w_{n-1}=w_\ell {\tt 1}E(v)$ is a suffix of $w_n$, we can see that $f$ has an occurrence in $w_n$ containing $|w_{n-1}|$.
        \item $\delta_n=\tt 0$ and $w_n=(w_{n-1}{\tt 0})^{E}=v {\tt 1} w_\ell {\tt 0} E(v)$,
            where $w_{\ell}$ is the longest antipalindromic prefix followed by $\tt 0$ in $w_{n-1}$ (it may be empty), and $v$ is a prefix of the word $w_n$. 
      The longest antipalindromic prefix of $w_n$ followed by $\tt 0$ is equal to $w_{n-1}$ and the longest antipalindromic prefix $q$ of $w_n$ followed by $\tt 1$ is the same as for $w_{n-1}$. We assume that $\{|q|, |w_\ell|, |w_{n-1}|-|q|-1\}$ is an attractor of $w_{n-1}$ and we want to show that $\Gamma=\{|q|, |w_{n-1}|, |w_n|-|q|-1\}$ is then an attractor of $w_n$. Below, we underline the positions of the attractor of $w_{n-1}$ and the positions of $\Gamma$ in $w_n$:
            \begin{displaymath}
            \begin{array}{rcl}
            w_{n-1}&=&q\underline{\tt 1}\cdots=\cdots \underline{\tt 0}q=w_{\ell}\underline{\tt 0}E(v)\,;\\
            w_n&=&q\underline{\tt 1}\cdots=\cdots \underline{\tt 0}q=v {\tt 1} w_\ell \underline{\tt 0} E(v)\,.
        \end{array}
        \end{displaymath}
        Each factor $f$ of $w_n$ either crosses $|w_{n-1}|$, \textit{i.e.}, the underlined $\tt 0$ in the last expression for $w_n$ above, or $f$ is contained in $w_{n-1}$. By the form of the attractor of $w_{n-1}$, the factor $f$ crosses $|q|$ or $|w_{n-1}|-|q|-1$ or $|w_\ell|$. In the last two cases, since $w_{n-1}=w_\ell {\tt 0}E(v)$ is a suffix of $w_n$, we can see that $f$ has an occurrence in $w_n$ containing $|w_n|-|q|-1$ or $|w_{n-1}|$.
    \end{enumerate}
\end{itemize}
Now let us show the attractor's minimality. 
 We could see in the above proof that if ${\tt 0}^{k}$ is a prefix of $\Delta$, then the attractor $\Gamma$ of $w_k$ from the theorem is of minimum size (equal to two). 
 \begin{itemize}
     \item As soon as ${\tt 0}^{k}\tt 1$ for $k\geq 2$ is a prefix of $\Delta$, then $w_{k+1}= ({\tt 01})^{k-1}{\tt 011001}({\tt 01})^{k-1}$. It follows from the following observation and Definition~\ref{def:generalized_pseudostandard} that for each $n\geq k+1$, the factor $\tt 00$, resp. $\tt 11$, only occurs as a factor of $\tt 011001$ in $w_n$. We can observe that the prefix ${\tt 0}^k \tt 10$ yields 
     \begin{displaymath}
     w_{k+2}= ({\tt 01})^{k-1}{\tt 011001}({\tt 01})^{k-1}{\tt 011001}({\tt 01})^{k-1}\,,
     \end{displaymath}
     and similarly, for the prefix ${\tt 0}^k \tt 11$, we obtain 
     \begin{displaymath}
     w_{k+2}= ({\tt 01})^{k-1}{\tt 011001}({\tt 01})^{k-2}{\tt 011001}({\tt 01})^{k-1}\,.
     \end{displaymath}
     Generally, 
     \begin{equation}\label{p_i}
         w_n=({\tt 01})^{p_0}{\tt 011001}({\tt 01})^{p_1}{\tt 011001}\cdots({\tt 01})^{p_{j-1}}{\tt 011001}({\tt 01})^{p_j}\,,
         \end{equation}
         where $p_i\in \{k-2, k-1\}$ and $p_0=p_j=k-1$. From this, we can also see that the factors ${\tt 011001}$ do not overlap anywhere in the word.
         
     An attractor of size two does not exist any more. Let us explain why. All factors of length two, \textit{i.e.}, $\tt 00, 01, 10, 11$, have to cross the attractor. We underline below the possible positions in $w_n$ from (\ref{p_i}) for an attractor of size two containing $\tt 00, 01, 10, 11$:
     \begin{enumerate}
     \item $w_{n}= \cdots ({\tt 01})^{p_i}{\tt 01\underline{1}0\underline{0}1}({\tt 01})^{p_{i+1}}\cdots$
     \item $w_{n}= \cdots ({\tt 01})^{p_i}{\tt 0\underline{1}1\underline{0}01}({\tt 01})^{p_{i+1}}\cdots$
     \item $w_{n}= \cdots ({\tt 01})^{p_i}{\tt 0110\underline{0}1}({\tt 01})^{p_{i+1}}\cdots ({\tt 01})^{p_{i+m}}{\tt 01\underline{1}001}({\tt 01})^{p_{i+m+1}}\cdots$
      \item $w_{n}= \cdots ({\tt 01})^{p_i}{\tt 011\underline{0}01}({\tt 01})^{p_{i+1}}\cdots ({\tt 01})^{p_{i+m}}{\tt 0\underline{1}1001}({\tt 01})^{p_{i+m+1}}\cdots$
     \item $w_{n}= \cdots ({\tt 01})^{p_i}{\tt 01\underline{1}001}({\tt 01})^{p_{i+1}}\cdots ({\tt 01})^{p_{i+m}}{\tt 0110\underline{0}1}({\tt 01})^{p_{i+m+1}} \cdots$
      \item $w_{n}= \cdots ({\tt 01})^{p_i}{\tt 0\underline{1}1001}({\tt 01})^{p_{i+1}}\cdots ({\tt 01})^{p_{i+m}}{\tt 011\underline{0}01}({\tt 01})^{p_{i+m+1} }\cdots$
     \end{enumerate}
     where $i, m \in \naturals, m\geq 1, i+m+1 \leq j$.
    However, in all the above cases, either $\tt 010$, or $\tt 101$ does not cross the attractor. 
 \item If ${\tt 0}{\tt 1}^k$ for $k\geq 1$ is a prefix of $\Delta$, then $w_{k+1}={\tt 01}{\tt \underline{1}0\underline{0}1}({\tt 1001})^{k-1}$, where we underlined the positions of an attractor of size two. In this case, the attractor $\Gamma$ from the theorem is not minimal, let us underline its positions: $w_{k+1}={\tt \underline{0}1}{\tt 1\underline{0}01}{(\tt 1001)}^{k-2}{\tt \underline{1}001}$ for $k\geq 2$ and $w_{k+1}={\tt \underline{0}1}{\tt \underline{1}\underline{0}01}$ for $k=1$.
  
 \item As soon as ${\tt 0}{\tt 1}^k\tt 0$ for $k\geq 1$ is a prefix of $\Delta$, then $w_{k+2}={\tt 01}({\tt 1001})^{k}({\tt 0110})^{k}\tt 01$. 
 It follows from Definition~\ref{def:generalized_pseudostandard} that for each $n\geq k+2$, the factor $\tt 00$, resp. $\tt 11$, only occurs as a factor of $\tt 011001$ in $w_n$. An attractor of size two does not exist -- the explanation is analogous as above.

 \end{itemize}
     
\end{proof}
\begin{corollary} Let $\uu$ be a pseudostandard sequence, i.e., $\uu=\uu(\Delta, E^{\omega})$.  
\begin{itemize}
\item If $\Delta\in \{{\tt 0}^{\omega}, {\tt 1}^{\omega}, {\tt 01^{\omega}}, {\tt 10^{\omega}}\}$, then all antipalindromic prefixes of $\uu$ have minimal attractors of size two. 
\item If $\Delta\not \in \{{\tt 0}^{\omega}, {\tt 1}^{\omega}, {\tt 01^{\omega}}, {\tt 10^{\omega}}\}$, then all sufficiently long antipalindromic prefixes of $\uu$ have minimal attractors of size three.
\end{itemize}
\end{corollary}

\begin{example}
Consider $\Delta=\tt 01001\cdots$. The first six prefixes of $\mathbf u(\Delta, E^{\omega})$ with the positions of attractor underlined read:\\
\begin{displaymath}
\begin{array}{rcll}
w_0&=&\varepsilon \\
w_1&=&\tt \underline{01}\\
w_2&=&\tt \underline{0}1\underline{10}01\\
w_3&=&\tt 01\underline{1}001\underline{0}11\underline{0}01\\
w_4&=&\tt 01\underline{1}001011001\underline{0}11\underline{0}01\\
w_5&=&\tt 011001011001\underline{0}11\underline{0}01\underline{1}001011001011001\,.\\
\end{array}
\end{displaymath}
\end{example}

\section{String attractors of complementary-symmetric Rote sequences}\label{sec:Rote}

This section is devoted to the study of attractors of pseudopalindromic prefixes of complementary-symmetric Rote sequences, which form a subclass of generalized pseudostandard sequences. However, besides being generalized pseudostandard sequences, they are also closely related to Sturmian sequences. 
\begin{definition}
 \emph{Rote sequences} are binary sequences having complexity $2n$ for each $n\geq 1$. A Rote sequence $\uu$ is called \emph{complementary-symmetric (CS)} if its language is closed under the letter exchange, \textit{i.e.}, for each factor $v=v_0 v_1\cdots v_{n-1}$ of $\uu$, the word $\overline{v}=\overline{v_0}\ \overline{v_1}\cdots \overline{v_{n-1}}$ is also a factor of $\uu$.    
\end{definition}
Let $u=u_0u_1\cdots u_{n-1}$ be a binary word on $\{\tt 0,1\}$ of length at least two. We denote by $S(u)$ the word $v = v_0v_1\cdots v_{n - 2}$ defined by
\[v_i = (u_{i+1} + u_i) \bmod 2 \quad \mbox{for $i = 0,1,\ldots,n-2$}.\]
For example, if $u=\tt 0011010$, then $S(u)=\tt 010111$.
The definition may be extended to sequences: if $\uu$ is a~sequence over $\{\tt 0,1\}$, then $S(\uu)$ denotes the sequence $\vv=v_0 v_1 v_2\cdots$, where 
\[v_i = (u_{i+1} + u_i) \bmod 2 \quad \mbox{for $i = 0,1,\ldots$}\]

CS Rote sequences are connected to Sturmian sequences by a structural theorem.

\begin{theorem}[Rote~\rm\cite{Rote}] \label{T:rote}
A binary sequence $\uu$ is a CS Rote sequence if and only if the sequence $S(\uu)$ is a Sturmian sequence.
\end{theorem}

We say that a CS Rote sequence $\uu$ is \emph{standard} if both ${\tt 0}\uu$ and ${\tt 1}\uu$ are CS Rote sequences. Equivalently, a sequence $\uu$ is standard CS Rote if and only if $S(\uu)$ is standard Sturmian.
The relation between pseudopalindromic prefixes of a standard CS Rote sequence $\uu$ and palindromic prefixes of a standard Sturmian sequence $S(\uu)$ is as follows.
\begin{lemma}[Lemma~37~\cite{MaPa}]\label{lem:pseudopalindromesRote}
Let $\uu$ be a standard CS Rote sequence. Let $u_0=\varepsilon, u_1, u_2, \dots$ be the palindromic prefixes of $S(\uu)$ ordered by length, and $w_0=\varepsilon, w_1, w_2, \dots$ be the pseudopalindromic prefixes of $\uu$ ordered by length. Then $S(w_{n+1})=u_n$ for all $n \in \naturals, n \geq 1$.
\end{lemma}

\begin{remark}
Let us explain that it is not possible to use known attractors of palindromic prefixes of standard Sturmian sequences to obtain attractors of pseudopalindromic prefixes of CS Rote sequences.

Consider the following palindromic prefix $u=\tt 010010010$ of a standard Sturmian sequence. The corresponding standard CS Rote sequence starting with $\tt 0$ has the antipalindromic prefix $w=\tt 0011100011$, \textit{i.e.}, $S(w)=u$. 
Let us underline the positions of the attractor of $u$ from Theorem~\ref{thm:sturmian_attractor}: $u=\tt 0\underline{1}0010\underline{0}10$ and also from \cite{Mantaci2021}(Theorem 22): $u=\tt 010010\underline{01}0$. Now, the factor $\tt 10$ has a~unique occurrence in $w$, therefore each attractor of $w$ has to contain either the position $4$ or $5$. However, there is no straightforward way to get such positions from the underlined attractors of $u$ (or their mirror image from Observation~\ref{rem:mirror}).
\end{remark}

Blondin-Massé at al. \cite{MaPa} showed that standard CS Rote sequences form a subclass of binary generalized pseudostandard sequences. Moreover, they described precisely the form of the corresponding directive bi-sequence.  
\begin{theorem}[\cite{MaPa}]\label{thm:Blondin-Masse} A sequence $\uu$ is a standard CS Rote sequence if and only if it is aperiodic and
$\uu=\uu(\Delta, \Theta)$ for some directive bi-sequence $(\Delta, \Theta)$ such that $\Theta$ starts with $R$ and no factor of length
two of the directive bi-sequence is in the set 
\begin{displaymath}
\left\{(ab,EE) : a,b \in \{\tt 0, 1\}\right\}\cup \left\{(a\overline{a},RR) : a \in \{\tt 0, 1\}\right\}\cup \left\{(aa,RE) : a \in \{\tt 0, 1\}\right\}\,.
\end{displaymath}
Moreover, the prefixes $w_n$ from Definition~\ref{def:generalized_pseudostandard} coincide with all pseudopalindromic prefixes of $\uu$.
\end{theorem}
The aperiodicity of a binary generalized pseudostandard sequence may be recognized easily.
\begin{theorem}[\cite{DvFl16}]\label{thm:periodicity}
Let $(\Delta, \Theta)$ be a directive bi-sequence. Then $\uu=\uu(\Delta, \Theta)$ is aperiodic if and only if there is no bijection $f:\{E, R\} \to \{\tt 0, 1\}$ such that $f(\vartheta_n)=\delta_{n+1}$ for all sufficiently large $n$.   
\end{theorem}

In the proof of the main theorem on string attractors of pseudopalindromic prefixes of standard CS Rote sequences, the following statements will be useful.
\begin{observation}\label{rem:mirror}
If $w$ is a pseudopalindrome with an attractor $\Gamma$, the mirror image \begin{math}\Gamma^R=\{|w|-1-\gamma \ : \gamma \in \Gamma\}\end{math} is an attractor of $w$, too. 
\end{observation}

\begin{lemma}\label{lem:Rote_pseudopalindrome}
Let $\uu$ be a standard CS Rote sequence and let $(w_n)_{n=1}^{\infty}$ be the sequence of all non-empty pseudopalindromic prefixes of $\uu$, ordered by length. 
Then for $n\geq 2$ and $a \in \{\tt 0,1\}$ such that $w_{n-1}a$ is a prefix of $w_n$, we have:
\begin{enumerate}
\item If $w_{n-1}=R(w_{n-1})$ and $w_n=R(w_n)$, then $w_n=w_{n-1}a\overline{w}$, where $w$ is the longest antipalindromic prefix of $w_n$ followed by ${a}$.
\item If $w_{n-1}=R(w_{n-1})$ and $w_n=E(w_n)$, then $w_n=w_{n-1}a\overline{w}$, where $w$ is the longest palindromic prefix of $w_n$ followed by $\overline{a}$.
\item  If $w_{n-1}=E(w_{n-1})$ and $w_n=R(w_n)$, then $w_n=w_{n-1}a{w}$, where $w$ is the longest palindromic prefix of $w_n$ followed by ${a}$.
\end{enumerate}
\end{lemma}
\begin{proof}
Assume without loss of generality that the Rote sequence starts with $\tt 0$.
The reader is invited to check the cases, where $S(w_n)$ contains only one letter, \textit{i.e.}, the cases where $w_n={\tt 0}^n$ or $w_n=({\tt 01})^{\frac{n}{2}}$ for $n$ even or $w_n=({\tt 01})^{\frac{n-1}{2}}\tt 0$ for $n$ odd. In the sequel, assume $S(w_n)$ contains both letters. The possible prefixes of $(\Delta, \Theta)$ are given in Theorem~\ref{thm:Blondin-Masse}.
   \begin{enumerate}
    \item Since $w_n=w_{n-1}aw$ and $w_{n-1}$ as a palindrome has $\tt 0$ as both the first and the last letter, then by Lemma~\ref{lem:pseudopalindromesRote}, Theorem~\ref{T:rote}, and Proposition~\ref{prop:Sturmian_palindromes}, we obtain $S(w_n)=u_n=u_{n-1}a\overline{a}u=S(w_{n-1})a\overline{a}S(w)$, where $u_n$ is the $n$-th palindromic prefix of the corresponding Sturmian sequence and $u$ is the longest palindromic prefix of $u_n$ followed by $\overline{a}$. Consequently, $w_n$ is equal to $w_{n-1}aw$, where $w$ starts with $\tt 1$ (since $S(aw)=\overline{a}u$), ends with $\tt 0$ (since $w_n$ ends with $\tt 0$) and is a pseudopalindrome by Lemma~\ref{lem:pseudopalindromesRote}. Thus $w=E(w)$. Therefore, since $w$ is a suffix of $w_n=R(w_n)$, we get $\overline{w}$ is an antipalindromic prefix of $w_n$ followed by ${a}$. Moreover, it is the longest antipalindromic prefix with this property because $S(\overline{w})=u$, where $u$ is the longest palindromic prefix of $S(w_n)$ followed by $\overline a$.   
\item The proof is similar as before. Since $w_n=w_{n-1}a w$ and $w_{n-1}$ as a~palindrome has $\tt 0$ as both the first and the last letter, then by Lemma~\ref{lem:pseudopalindromesRote}, Theorem~\ref{T:rote}, and Proposition~\ref{prop:Sturmian_palindromes}, we obtain $S(w_n)=u_{n-1}a\overline{a}u$, where $u$ is the longest palindromic prefix of $u_n$ followed by $\overline{a}$. Consequently, $w_n$ is equal to $w_{n-1}aw$, where $w$ starts with $\tt 1$ (since $S(aw)=\overline{a}u)$, ends with $\tt 1$ (since $w_n$ ends with $\tt 1$) and is a pseudopalindrome by Lemma~\ref{lem:pseudopalindromesRote}. Therefore, since $w$ is a suffix of $w_n=E(w_n)$, we get $\overline{w}$ is the longest palindromic prefix of $w_n$ followed by $\overline{a}$.  
\item Since $w_n=w_{n-1}aw$ and $w_{n-1}$ as an antipalindrome has $\tt 1$ as the last letter, then by Lemma~\ref{lem:pseudopalindromesRote}, Theorem~\ref{T:rote}, and Proposition~\ref{prop:Sturmian_palindromes}, we obtain $S(w_n)=u_{n-1}\overline{a}au$, where $u$ is the longest palindromic prefix of $u_n$ followed by ${a}$. Consequently, $w_n$ is equal to $w_{n-1}aw$, where $w$ starts with $\tt 0$, ends with $\tt 0$ and is a pseudopalindrome by Lemma~\ref{lem:pseudopalindromesRote}. Therefore, since $w$ is a suffix of $w_n=R(w_n)$, we get $w$ is the longest palindromic prefix of $w_n$ followed by ${a}$.  

\end{enumerate}    
\end{proof}

Let us state the main theorem describing the minimal string attractors of pseudopalindromic prefixes of standard CS Rote sequences. 
\begin{theorem}\label{thm:Rote_atrractor}
Let $\uu$ be a standard CS Rote sequence, then the size of the minimal attractor of any pseudopalindromic prefix equals the number of distinct letters contained in the prefix.  
More precisely, let $(w_n)_{n=1}^{\infty}$ be the sequence of all non-empty pseudopalindromic prefixes of $\uu$ ordered by length and consider and consider $w_n$ containing both letters and let $w_{n-1}a$ be a prefix of $w_n$, where $a \in \{\tt 0,1\}$. Then the minimal attractor of the pseudopalindromic prefix $w_n$ is of the following form:
\begin{enumerate}
    \item If $w_n=E(w_n)$ and $w$ is the longest antipalindromic prefix of $w_n$ followed by $\overline{a}$, then 
    \begin{displaymath}
    \Gamma=\{|w|, |w_{n-1}|\}
    \end{displaymath}
    is an attractor of $w_n$.
    \item If $w_n=R(w_n)$, $w_{n-1}=E(w_{n-1})$, and $w$ is the longest palindromic prefix of $w_n$ followed by $\overline{a}$, then 
    \begin{displaymath}
    \Gamma=\{|w|, |w_{n-1}|\}
    \end{displaymath}
    is an attractor of $w_n$.
    \item If $w_n=R(w_n)$, $w_{n-1}=R(w_{n-1})$, and $m$ is the minimum index satisfying that $w_i=R(w_i)$ for all $i \in \{m,\dots, n\}$, then the attractor of $w_m$ from Item 2 is an attractor of $w_n$. 
\end{enumerate}
    
\end{theorem}
\begin{proof} First of all, Theorem~\ref{thm:Blondin-Masse} describes the form of the unique bi-sequence $(\Delta, \Theta)$ satisfying that the pseudopalindromic prefixes $w_n$ of $\uu$ correspond to the prefixes $w_n$ given by Definition~\ref{def:generalized_pseudostandard}. It follows that $\Theta$ has to start with $R$. Let us assume without loss of generality that $({\tt 0}, R)$ is the first element of $(\Delta, \Theta)$. 
If a pseudopalindromic prefix contains one letter, then any position is its attractor. Further on, let us consider pseudopalindromic prefixes $w_n$ containing two distinct letters. Let us proceed by induction on $n$. 

\medskip

Consider the first pseudopalindromic prefix $w_{k+1}$ containing both letters $\tt 0$ and $\tt 1$. Then by Theorem~\ref{thm:Blondin-Masse} $({\tt 0}^{k}{\tt 1}, R^{k}E)$ is a prefix of $(\Delta, \Theta)$ and $k\geq 1$. Then $w_{k+1}={\tt 0}^{k}{\tt 1}^{k}$ and the longest antipalindromic prefix of $w_{k+1}$ followed by $\tt 0$ is $w_0=\varepsilon$. Hence 
$\Gamma=\{0, k\}$ is clearly an attractor of $w_{k+1}=\underline{\tt 0}{\tt 0}^{k-1}\underline{\tt 1}{\tt 1}^{k-1}$ (we underlined the positions of $\Gamma$).

\medskip
Assume that for some $n\geq k+1$, we have $w_n=E(w_n)$ and $w_{n-1}a$ is a prefix of $w_n$ for $a \in \{\tt 0,1\}$ and the claim on the attractor holds, \textit{i.e.}, $w_n$ has the attractor 
\begin{equation}\label{eq:attractor}
    \begin{array}{rcl}
    \Gamma_1&=&\{|w_i|, |w_{n-1}|\}\,,
    \end{array} 
\end{equation}
    where $w_i$ is the longest antipalindromic prefix of $w_n$ followed by $\overline{a}$.
    
Let us assume $\vartheta_n=\vartheta_{n+m+1}=E$ (by Theorems~\ref{thm:Blondin-Masse} and~\ref{thm:periodicity} such an integer $m$ exists and $m\geq 1$), while $\vartheta_\ell=R$ for all $\ell \in \{n+1, \dots, n+m\}$. 
We will show that under this assumption, $w_{n+1}$ up to $w_{n+m+1}$ have also the attractors as described in Theorem~\ref{thm:Rote_atrractor}. This will prove the theorem completely. 
 
\medskip

\noindent There are four situations to be considered according to Theorem~\ref{thm:Blondin-Masse}. 
\begin{enumerate}
\item $m=1$ and $(\delta_{n-1}\delta_n\delta_{n+1}\delta_{n+2}, \vartheta_{n-1}\vartheta_n\vartheta_{n+1}\vartheta_{n+2})=(\overline{a}aa\overline{a}, RERE)$;
 \item $m\geq 2$ and $(\delta_{n-1}\delta_n\delta_{n+1}\cdots\delta_{n+m}\delta_{n+m+1}, \vartheta_{n-1}\vartheta_n\vartheta_{n+1}\cdots\vartheta_{n+m}\vartheta_{n+m+1})=(\overline{a}a^{m+1}\overline{a}, RER^mE)$;  
\item $m=1$ and $(\delta_{n-1}\delta_n\delta_{n+1}\delta_{n+2}, \vartheta_{n-1}\vartheta_n\vartheta_{n+1}\vartheta_{n+2})=(\overline{a}a\overline{a}a, RERE)$;
    \item $m\geq 2$ and $(\delta_{n-1}\delta_n\delta_{n+1}\cdots\delta_{n+m}\delta_{n+m+1}, \vartheta_{n-1}\vartheta_n\vartheta_{n+1}\cdots\vartheta_{n+m}\vartheta_{n+m+1})=(\overline{a}a\overline{a}^m a, RER^mE)$.
\end{enumerate}
We will treat the first two of them. The remaining ones are analogous. In both cases, using Lemma~\ref{lem:Rote_pseudopalindrome}, we have
\begin{equation}\label{eq:1}
w_n=w_{n-1}a\overline{w_j}\,,
\end{equation}
\begin{equation}\label{eq:2}
w_n=E(w_n)=w_j\overline{a}\overline{w_{n-1}}\,,    
\end{equation}
where $w_j$ is the longest palindromic prefix of $w_n$ followed by $\overline{a}$.
Since $w_n=(w_{n-1}a)^E$, it follows by (\ref{eq:1}) and (\ref{eq:2}) that $w_{n-1}=w_j\overline{a}x$, where $x$ is the longest antipalindromic suffix of $w_{n-1}$ preceded by $\overline{a}$, or equivalently, $\overline{x}$ is the longest antipalindromic prefix of $w_{n-1}$ followed by $\overline{a}$, \textit{i.e.}, $\overline{x}=w_i$, as defined in~(\ref{eq:attractor}). Hence, $w_{n-1}=w_j\overline{a}\overline{w_i}=w_i\overline{a}w_j$, where we used palindromicity of $w_{n-1}$ and $w_j$ and antipalindromicity of $w_i$ in the last equality.
Therefore, we get the following expressions for $w_n$, where we underlined the positions of the attractor $\Gamma_1$ of $w_n$ from~(\ref{eq:attractor}) (the first line) and the mirror image attractor  ${\Gamma_1}^R$ from Observation~\ref{rem:mirror} (the second line): 
    \begin{equation}\label{eq:m1n:1}
   \begin{array}{rcl} 
   w_n&=&w_{n-1}a\overline{w_j}=w_j\overline{a}\overline{w_i}a\overline{w_j}=w_i\underline{\overline{a}}w_j\underline{a}\overline{w_j}\,,\\
   w_n&=&w_j\overline{a}\overline{w_{n-1}}=w_j\underline{\overline{a}}\overline{w_j}\underline{a}{w_i}=w_j{\overline{a}}\overline{w_i}a\overline{w_j}\,.
        \end{array}
        \end{equation}
      
\begin{enumerate}
    \item $m=1$ and $(\delta_{n-1}\delta_n\delta_{n+1}\delta_{n+2}, \vartheta_{n-1}\vartheta_n\vartheta_{n+1}\vartheta_{n+2})=(\overline{a}aa\overline{a}, RERE)$:
    \begin{itemize}
        \item Using the definition of palindromic closure and (\ref{eq:1}), we obtain
    \begin{equation}\label{eq:m1n+1:1}
        w_{n+1}=(w_na)^R=w_{n-1}a\overline{w_j}aw_{n-1}=w_n a w_{n-1}\,.
    \end{equation}
    We will show that $\Gamma_2=\{|w_j|, |w_n|\}$ is an attractor of $w_{n+1}$ -- see the corresponding positions underlined (we rewrite $w_n$ by (\ref{eq:2})):
    \begin{equation}\label{eq:m1attractorn+1}
    w_{n+1}=w_j\underline{\overline{a}}\overline{w_{n-1}}\underline{a}w_{n-1}\,.
    \end{equation}
    By~(\ref{eq:m1n+1:1}), each factor $f$ of $w_{n+1}$ either crosses $|w_n|$, \textit{i.e.}, the underlined $a$ in~(\ref{eq:m1attractorn+1}),  or is entirely contained in $w_n$. In that case, consider the mirror image attractor ${\Gamma_1}^R=\{|w_j|, |w_n|-|w_i|-1\}$ of $w_n$. Using (\ref{eq:m1n:1}) we can rewrite $w_{n+1}$ as
    \begin{equation}\label{eq:m1n+1:2}
        w_{n+1}=\overunderbraces{&&\br{1}{\overline{w_{n-1}}}&&}{&w_j\underline{\overline{a}}&\overline{w_j}aw_i&\underline{a}w_{n-1}}{&\br{2}{w_n}}
        =\overunderbraces{&&\br{2}{\overline{w_{n-1}}}}{&w_j\underline{\overline{a}}\overline{w_i}a&\overline{w_j}&\underline{a}w_i&\overline{a}w_j}{&\br{2}{{w_{n}}}}\,,
    \end{equation}
    where we underlined the positions from $\Gamma_2$.
    From this form, we can see that $f$ contained in $w_n$ either crosses $|w_j|$, \textit{i.e.}, the underlined $\overline{a}$, or $f$ crosses the position $|w_{n}|-|w_i|-1$ in~(\ref{eq:m1n+1:2}) and is contained in the word $\overline{w_{n-1}}=\overline{w_j}\underline{a} w_i$, where we underlined the position crossed by $f$ in $\overline{w_{n-1}}$. Observing the right-hand form of $w_{n+1}$ in~(\ref{eq:m1n+1:2}), the factor $f$ has then an occurrence containing the position $|w_n|$ in $w_{n+1}$, \textit{i.e.}, the underlined $a$.
   \item 
    Next, using antipalindromic closure, we obtain
    \begin{equation}\label{eq:m1n+2:1}
      w_{n+2}=(w_{n+1}\overline{a})^E=\overbrace{w_{n-1}a\overline{w_n}}^{w_{n+1}}\overline{a}\overline{w_{n-1}}\,.  
    \end{equation}
    We will show that $\Gamma_3=\{|w_n|, |w_{n+1}|\}$ is an attractor of $w_{n+2}$. Each factor $f$ of $w_{n+2}=w_n\underline{a}w_{n-1}\underline{\overline{a}}\overline{w_{n-1}}$ (we underlined the positions of $\Gamma_3$) either crosses $|w_{n+1}|$, \textit{i.e.}, the underlined $\overline{a}$, or is entirely contained in $w_{n+1}=w_n{a}w_{n-1}$. In this case, we can write $w_{n+2}$ as
        \begin{equation}\label{eq:m1n+2:2}
            w_{n+2}=\overunderbraces{&\br{5}{w_{n+1}}}{&w_j\overline{a}&\overline{w_{n-1}}&\underline{a}&w_i\overline{a}&{w_j}&\underline{\overline{a}}&\overline{w_{n-1}}}{&\br{2}{w_n}&&\br{2}{w_{n-1}}&&}\,.
        \end{equation}
        
        By~(\ref{eq:m1attractorn+1}), $f$ then either crosses $|w_n|$, \textit{i.e.}, the underlined ${a}$, or $f$ is contained in $w_n=w_j\overline{a}\overline{w_{n-1}}$ and crosses $|w_j|$. However, since $w_n$ forms also a suffix of $w_{n+2}$, the factor $f$ has an occurrence in $w_{n+2}$ containing the position $|w_{n+1}|$, \textit{i.e.}, the underlined $\overline{a}$. 

    \end{itemize}
    \item $m\geq 2$ and $(\delta_{n-1}\delta_n\delta_{n+1}\cdots\delta_{n+m}\delta_{n+m+1}, \vartheta_{n-1}\vartheta_n\vartheta_{n+1}\cdots\vartheta_{n+m}\vartheta_{n+m+1})=(\overline{a}a^{m+1}\overline{a}, RER^mE)$: The proof that $\Gamma_2=\{|w_j|, |w_n|\}$ is an attractor of $w_{n+1}$ stays the same as above. 
    
    \begin{itemize}
        \item Using the definition of palindromic closure, since $w_{n-1}$ is the longest palindromic prefix of $w_{n+1}$ followed by $a$, we obtain
    \begin{equation}\label{eq:m2n+2}
     w_{n+2}=(w_{n+1}a)^R=w_naw_{n-1}a\overline{w_n}=w_naw_{n+1}\,. 
    \end{equation}
    We will show that $\Gamma_2=\{|w_j|, |w_n|\}$ is an attractor of $w_{n+2}$. Indeed, each factor $f$ either crosses $|w_n|$ or by~(\ref{eq:m2n+2}) $f$ is entirely contained in $w_{n+1}$. Since $w_{n+1}$ is a prefix of $w_{n+2}$, the factor $f$ has an occurrence containing an element of $\Gamma_2$.\\
    \noindent Similarly, for $k \in \{3,\dots, m\}$, we have $w_{n+k}=w_naw_{n+k-1}$.
    The attractor of $w_{n+k}$ is again equal to $\Gamma_2$: each factor $f$ either crosses $|w_n|$, or $f$ is entirely contained in the prefix $w_{n+k-1}$ of $w_{n+k}$ and the attractor of $w_{n+k-1}$ is by induction assumption $\Gamma_2=\{|w_j|, |w_n|\}$. 
    \item
    Using the definition of the antipalindromic closure, since $w_n$ is the longest antipalindromic prefix followed by $a$, we have
   \begin{displaymath}
   \begin{array}{rcl}
    w_{n+m+1}&=&(w_{n+m}\overline{a})^E=w_{n+m-1}a\overline{w_n}\overline{a}\overline{w_{n+m-1}}=w_{n+m}\overline{a}\overline{w_{n+m-1}}\,.
    \end{array}
    \end{displaymath}
    We will show that $\Gamma_3=\{|w_n|, |w_{n+m}|\}$ is an attractor of $w_{n+m+1}$. Each factor $f$ of $w_{n+m+1}=w_n\underline{a}w_{n+m-1}\underline{\overline{a}}\overline{w_{n+m-1}}$ (we underlined the positions of the expected attractor $\Gamma_3$) either crosses $|w_{n+m}|$, \textit{i.e.}, the underlined $\overline{a}$, or $f$ is entirely contained in $w_{n+m}$ or in $\overline{w_{n+m-1}}$. The word $w_{n+m+1}$ can be expressed as
        \begin{align}
            w_{n+m+1}&=\overunderbraces{&\br{3}{w_{n+m-1}}&&&&\br{3}{w_n}}{&w_n&\underline{a}&w_{k}&a&\overline{w_{n-1}}&\overline{a}&w_j&\overline{\underline{a}}&\overline{w_{n-1}}&\overline{a}&w_j&\overline{a}&\overline{w_{k}}}{&\br{7}{w_{n+m}}&&\br{5}{\overline{w_{n+m-1}}}}\,, \label{eq:m2n+m+1:1}
            \end{align}
            \begin{align}
           w_{n+m+1}&=\overunderbraces{&\br{1}{w_{n+m-1}}&&\br{1}{\overline{w_n}}}{&w_{n-1}a\overline{w_j}\underline{a}w_k&{a}&\overline{w_{n-1}}\overline{a}w_j&\overline{\underline{a}}\overline{w_{k}}&\overline{a}w_n}{&&&\br{2}{\overline{w_{n+m-1}}}}\,, \label{eq:m2n+m+1:2}
            \end{align}
            
        where $k=n+m-2$ if $m\geq 3$ and $k=n-1$ if $m=2$.
        
        \noindent If $f$ is contained in $w_{n+m}$, then $f$ crosses the attractor $\Gamma_2=\{|w_j|, |w_n|\}$ of $w_{n+m}$. It means that either $f$ crosses $|w_n|$, \textit{i.e.}, the underlined $a$ (see (\ref{eq:m2n+m+1:1})), or $f$ is contained in $w_n=w_j\underline{\overline{a}}\overline{w_{n-1}}$, where we underlined the position in $w_n$ crossed by $f$. Then by (\ref{eq:m2n+m+1:1}) the factor $f$ has an occurrence in $w_{n+m+1}$ containing $|w_{n+m}|$, \textit{i.e.}, the underlined $\overline{a}$. \\
        \noindent 
        If $f$ is contained in $\overline{w_{n+m-1}}$, then $f$ crosses its attractor $\Gamma_2=\{|w_j|, |w_n|\}$ (taking the positions only in $\overline{w_{n+m-1}}$). By~(\ref{eq:m2n+m+1:2}) the factor $f$ either crosses $|w_n|$ in $\overline{w_{n+m-1}}$, which means that $f$ crosses $|w_{n+m}|$ in $w_{n+m+1}$, \textit{i.e.}, the underlined $\overline{a}$, or $f$ is contained in $\overline{w_n}=\overline{w_j}\underline{a}w_{n-1}$ (a prefix of $\overline{w_{n+m-1}}$), where we underlined the position crossed by $f$. However, then $f$ has an occurrence in $w_{n+m+1}$ containing $|w_n|$, \textit{i.e.}, the underlined $a$, as can be observed from (\ref{eq:m2n+m+1:2}) (since $w_{n-1}$ is a prefix of $w_k$).
        
    \end{itemize}
    
\end{enumerate}
\end{proof}
\begin{example}
Let us consider a standard CS Rote sequence with the bi-sequence starting with\\ $({\tt 0011001}, RRERERE)$, \textit{i.e.}, corresponding to the situation of $m=1$ treated in the proof of Theorem~\ref{thm:Rote_atrractor}. The attractors' positions given in Theorem~\ref{thm:Rote_atrractor} for prefixes $w_n$ containing both letters are underlined.
    \begin{align*}
        w_1&=\tt {0}\\
        w_2&=\tt 0{0}\\
        w_3&=\overbrace{\tt \underline{0}0}^{w_2}{\tt \underline{1}1}\\
        w_4&=\overunderbraces{&\br{1}{w_1}}{&\tt 0&\tt \underline{0}11&\tt \underline{1}00}{&\br{2}{w_3}}\\
        w_5&=\overunderbraces{&\br{1}{w_3}}{&\tt 0011&\tt \underline{1}00&\tt \underline{0}11}{&\br{2}{w_4}}\\
        w_6&=\overunderbraces{&\br{1}{w_2}}{&\tt 00&\tt \underline{1}1100011&\tt \underline{0}0011100}{&\br{2}{w_5}}\\
        w_7&=\overunderbraces{&\br{1}{w_5}}{&\tt 0011100011&\tt \underline{0}0011100&\tt\underline{1}1100011}{&\br{2}{w_6}}
    \end{align*}
    \noindent For illustration of the case $m\geq 2$ from the proof of Theorem~\ref{thm:Rote_atrractor}, let us consider a bi-sequence starting with $({\tt 001100001}, RRERERRRE)$. The steps for $n\leq 6$ are identical with the previous example.
    \begin{align*}
        w_1&=\tt {0}\\
        w_2&=\tt 0{0}\\
        w_3&=\tt \underline{0}0\underline{1}1\\
        w_4&=\tt 0\underline{0}11\underline{1}00\\
        w_5&=\tt 0011\underline{1}00\underline{0}11\\
        w_6&=\overunderbraces{&\br{1}{w_2}}{&\tt 00&\tt \underline{1}1100011&\tt \underline{0}0011100}{&\br{2}{w_5}}\\
        w_7&=\overunderbraces{&\br{1}{w_2}}{&\tt 00&\tt \underline{1}1100011&\tt \underline{0}001110001100011100}{&\br{2}{w_5}}\\
        w_8&=\overunderbraces{&\br{1}{w_2}}{&\tt 00&\tt \underline{1}1100011&\tt \underline{0}00111000110001110001100011100}{&\br{2}{w_5}}\\
        w_9&=\overunderbraces{&\br{1}{w_5}}{&\tt 0011100011&\tt\underline{0}00111000110001110001100011100&\tt\underline{1}11000111001110001110011100011}{&\br{2}{w_8}}
    \end{align*}
\end{example}

\section{Open problems}
It still remains an open problem to find string attractors of pseudopalindromic prefixes $w_n$ of binary generalized pseudostandard sequences.
Recall partial steps done in this paper and also those done previously.  
For standard Sturmian sequences, the minimal attractors of factors containing two letters are of size two. For pseudostandard sequences we have described attractors of size three for antipalindromic prefixes and we have shown that, up to an exceptional case, they are minimal.
For pseudopalindromic prefixes of standard CS Rote sequences we have found attractors of size two. 
In both previous cases, we studied only pseudopalindromic prefixes, neither prefixes nor factors in general.
Even if we keep restricting our focus to pseudopalindromic prefixes of binary generalized pseudostandard sequences, the attractor may be larger. For instance, the minimal attractors of pseudopalindromic prefixes of length at least 8 of the Thue-Morse sequence are of size four~\cite{Kutsukake2020}.
Based on computer experiments, we conjecture that the size of minimal string attractors of pseudopalindromic prefixes of binary generalized pseudostandard sequences is of size at most four. 
An even more demanding open problem is to study attractors of $d$-ary generalized pseudostandard sequences for $d>2$ (defined in~\cite{LuDeLu} and studied in~\cite{JaPeSt2014}).

\section{Acknowledgements}
We would like to thank Karel Břinda for his great help with the programming part. We address our cordial thanks also to Francesco Dolce and Edita Pelantová for their careful reading of our manuscript. Last, but not least, we express our thanks to two anonymous referees whose thorough reports helped improve the paper.

\bibliographystyle{abbrv}
\bibliography{biblio}

\begin{thebibliography}{10}

\bibitem{Lothaire}
J.~Berstel and P.~S\'e\'ebold.
\newblock Sturmian words.
\newblock In M.~Lothaire, editor, {\em Algebraic Combinatorics on Words},
  volume~90 of {\em Encyclopedia of Mathematics and Its Applications}, pages
  45--110. Cambridge University Press, 2002.

\bibitem{MaPa}
A.~Blondin-Mass\'e, G.~Paquin, H.~Tremblay, and L.~Vuillon.
\newblock On generalized pseudostandard words over binary alphabet.
\newblock {\em J. Integer Seq.}, 16:Article 13.2.11, 2013.

\bibitem{LuDeLu}
A.~de~Luca and A.~D. Luca.
\newblock Pseudopalindrome closure operators in free monoids.
\newblock {\em Theoretical Computer Science}, 362(1):282--300, 2006.

\bibitem{LuMi94}
A.~{de Luca} and F.~Mignosi.
\newblock Some combinatorial properties of {S}turmian words.
\newblock {\em Theoretical Computer Science}, 136(2):361--385, 1994.

\bibitem{Dolce2023}
F.~Dolce.
\newblock String attractors for factors of the {T}hue-{M}orse word.
\newblock In A.~Frid and R.~Merca{\c{s}}, editors, {\em Combinatorics on
  Words}, pages 117--129, Cham, 2023. Springer Nature Switzerland.

\bibitem{Dv2023}
L.~Dvořáková.
\newblock String attractors of episturmian sequences.
\newblock {\em Theoretical Computer Science}, 986:114341, 2024.

\bibitem{DvFl16}
L.~Dvořáková and J.~Florian.
\newblock On periodicity of generalized pseudostandard words.
\newblock {\em Electronic Journal of Combinatorics}, 23:P1.2, 2016.

\bibitem{GhRoSt2023}
F.~Gheeraert, G.~Romana, and M.~Stipulanti.
\newblock String attractors of fixed points of k-bonacci-like morphisms.
\newblock In A.~Frid and R.~Merca{\c{s}}, editors, {\em Combinatorics on
  Words}, pages 192--205, Cham, 2023. Springer Nature Switzerland.

\bibitem{JaPeSt2014}
T.~Jajcayová, E.~Pelantová, and {\v S}.~Starosta.
\newblock Palindromic closures using multiple antimorphisms.
\newblock {\em Theoretical Computer Science}, 533:37--45, 2014.

\bibitem{KempaPrezza2018}
D.~Kempa and N.~Prezza.
\newblock At the roots of dictionary compression: String attractors.
\newblock In {\em Proceedings of the 50th Annual ACM SIGACT Symposium on Theory
  of Computing}, STOC 2018, page 827–840, New York, NY, USA, 2018.
  Association for Computing Machinery.

\bibitem{Kociumaka2021}
T.~Kociumaka, G.~Navarro, and N.~Prezza.
\newblock Towards a definitive measure of repetitiveness.
\newblock In {\em LATIN 2020: Theoretical Informatics: 14th Latin American
  Symposium, S\~{a}o Paulo, Brazil, January 5-8, 2021, Proceedings}, page
  207–219, Berlin, Heidelberg, 2021. Springer-Verlag.

\bibitem{Kutsukake2020}
K.~Kutsukake, T.~Matsumoto, Y.~Nakashima, S.~Inenaga, H.~Bannai, and M.~Takeda.
\newblock On repetitiveness measures of {Thue}-{Morse} words, 2020.

\bibitem{Mantaci2021}
S.~Mantaci, A.~Restivo, G.~Romana, G.~Rosone, and M.~Sciortino.
\newblock A combinatorial view on string attractors.
\newblock {\em Theoretical Computer Science}, 850:236--248, 2021.

\bibitem{Rote}
G.~Rote.
\newblock Sequences with subword complexity $2n$.
\newblock {\em Journal of Number Theory}, 46:196--213, 1993.

\bibitem{Shallit2021}
L.~Schaeffer and J.~Shallit.
\newblock String attractors for automatic sequences, 2020.

\end{thebibliography}
\label{sec:biblio}

\end{document}